\documentclass[12pt,a4paper]{amsart}

\usepackage{amsmath, amsfonts, amssymb, amsthm ,mathrsfs,bbm,units}

\newtheorem{theorem}{Theorem}[section]

\newtheorem{corollary}[theorem]{Corollary}

\newtheorem{proposition}[theorem]{Proposition}

\theoremstyle{remark}
\newtheorem{remark}[theorem]{Remark}

\usepackage{delarray}
\usepackage{enumerate}

\numberwithin{equation}{section}
\numberwithin{theorem}{section}
\numberwithin{figure}{section}

\newcommand{\ud }{\,{\rm d}}

\newcommand{\bbfR}{{\mathbb R}}

\newcommand{\vf}{{\varphi}}
\newcommand{\ve}{{\varepsilon }}
\newcommand{\se}{{\text{\rm{e}}}}

\newcommand{\R}{\mathbb{R}}

\newcommand{\un}{\mathbf{1}\!\!{\rm I}} 
\newcommand{\be}{\begin{equation}} 
\newcommand{\ee}{\end{equation}}
\newcommand{\bea}{\begin{eqnarray}} 
\newcommand{\eea}{\end{eqnarray}}
\newcommand{\bean}{\begin{eqnarray*}} 
\newcommand{\eean}{\end{eqnarray*}}
\newcommand{\rf}[1]{(\ref {#1})}

\def\dy{\,{\rm d}y}

\def\dta{\,{\rm d}\tau}

\def\e{\varepsilon}

\def\s{\sigma}

\def\g{\gamma}

\def\xn{|\!|\!|}

\def\mn{|\!\!|}

\def\mn2{|\!\!|_{M^{d(p-1)/2}}}

\def\a{\alpha}

\def\ap{\frac{\a}{p-1}}
\def\ts{\tilde s}

\title[Nonlocal nonlinear heat equation]{Around  a singular solution of a nonlocal \\ nonlinear heat equation}

\author[P. Biler]{Piotr Biler}
\address{\small Instytut Matematyczny, Uniwersytet Wroc\l awski,
 pl. Grunwaldzki 2/4, \hbox{50-384} Wroc\-\l aw, Poland}
\email{Piotr.Biler@math.uni.wroc.pl}

\author[D. Pilarczyk]{Dominika Pilarczyk}
\address{\small Wydzia{\l} Matematyki, Politechnika Wroc{\l}awska,
Wybrze\. ze Wyspia\'nskiego 27, 50-370 Wroc{\l}aw, Poland} 
\email{Dominika.Pilarczyk@pwr.edu.pl}

\begin{document}

\begin{abstract}
We study the existence of global-in-time solutions for a nonlinear heat equation with nonlocal diffusion, power nonlinearity and suitably small data (either  compared pointwisely to the singular solution or  in the norm of a critical Morrey space). 
Then, asymptotics of subcritical solutions is determined. 
These results are compared with conditions on the initial data leading to a finite time blowup. 
\end{abstract}

\date{\today}



\thanks{The authors thank Grzegorz Karch and Philippe Souplet for interesting discussions. 
The first author has been partially  supported by the NCN grant {2016/23/B/ST1/00434}. }

\maketitle

\section{Introduction and main results}

\setcounter{equation}{0}

Nonlinear evolution problems involving fractional Laplacian describing the anomalous diffusion (or the $\alpha $-stable L\'evy diffusion) have been extensively studied in the mathematical and physical literature, see \cite{Su} for the Cauchy problem \rf{f_lap}--\rf{u_0}, and \cite{BFW,BKW1,BKW2,BIK} for other examples of problems and for extensive list of references. 
One of these models is the following initial value problem for the reaction-diffusion equation with the anomalous diffusion
\begin{align}
   \label{f_lap} u_t &= - (-\Delta )^{\nicefrac{\alpha }{2}} u+ |u|^{p-1}u , \qquad \bbfR^d \times (0, \infty ),\\
   \label{u_0} u(x,0)&=u_0(x),
\end{align}
where the pseudodifferential operator $(- \Delta )^{\nicefrac{\alpha }{2}}$ with $0< \alpha \leqslant 2$ is defined by the Fourier transformation: $\widehat{(-\Delta )^{\nicefrac{\alpha }{2}} }(\xi ) = | \xi |^\alpha \widehat{u}(\xi )$. We will  use the following well-known integral definition of the fractional Laplacian with $\alpha\in(0,2)$, see e.g. \cite[formula (1.4)]{BKZ2},  
\be
-(-\Delta)^{\nicefrac{\alpha}{2}}f(x)={\mathcal A} 
\lim_{\delta\searrow0}\int_{\{|y|>\delta\}}\frac{f(x-y)-f(x)}{|y|^{d+\alpha}}\dy,\label{fr-lap}
\ee   
where the constant is, by e.g. \cite[{formula} (1.5)]{BKZ2},   
\be
{\mathcal A}={\mathcal A}({d,\alpha})=\frac{2^\alpha\Gamma\left(\frac{d+\alpha}{2}\right)}{\pi^{\frac{d}{2}}\left|\Gamma\left(-\frac{\alpha}{2}\right)\right|}.\label{stala-A}
\ee  
Moreover, we assume that $p>1$ and $u_0(x)\geqslant 0$.\\

This is a straightforward generalization of extensively studied classical nonlinear heat equation, see \cite{QS}, to the case of nonlocal but linear diffusion operators defined by fractional powers of Laplacian. 
Study of solutions around a critical singular solution which is not smoothed out by the diffusion is, in a sense, parallel to analyses for nonlinear heat equation in \cite{QS} and, in particular, in \cite{Pi1,Pi2}. 
This also  reveals some threshold phenomena as is in the former case.


\subsection*{Statement of results}

The form of singular steady state $u_\infty$ of equation \rf{f_lap} is given in Proposition \ref{singC}. 
Existence of global-in-time solutions for initial data in a suitable critical Morrey space is shown in Proposition \ref{glo1}. 
Local-in-time existence of solutions for initial data having local singularities weaker than $u_\infty$ is proved in Proposition \ref{MsMts}. 

Global-in-time existence of solutions for subcritical initial data $u_0\leqslant u_\infty$ is derived in Theorem \ref{glo2}, and this construction is based on a novel comparison principle for special solutions in Proposition \ref{glo3}. 
These solutions have diffusion dominated asymptotics.

Analysis of the linearization operator (nonlocal diffusion $+$ Hardy-type potential) around the singular state in Section \ref{asym} leads to its nonlinear $L^2$-asymptotic stability. Fine asymptotics of solutions in vicinity of $u_\infty$ is given. 

Finally, we interpret results on finite time blowup of solutions with ``large'' initial data (in \cite{Su} and recently \cite{B-bl-a}) as a kind of dichotomy, see Corollary \ref{dich}.

Concerning the question of nonexistence of global-in-time solutions to equation \rf{f_lap}, the first results have been proved in \cite{Su} with the argument based on the seminal idea of \cite{Fu}.
Extensions of such blowup results for more general equations with linear but nonlocal diffusion operators more general than fractional Laplacians and localized source terms are in a forthcoming paper \cite{B-bl-a} which improves some results in \cite{A} and clarifies sufficient conditions for blowup. 
Interpretations of sufficient blowup conditions in \cite{B-bl-a} lead to Corollary \ref{dich} which shows that the discrepancy bounds between sufficient bounds for global-in-time existence and sufficient bounds for blowup (for the same quantity like the Morrey space $M^{d(p-1)/\alpha}(\R^d)$ norm) are well controlled for large dimensions $d$.

 \subsection*{Notation}  
In the sequel, $\| \cdot \|_p$ denotes the usual $L^p(\mathbb R^d)$ norm, and $C$'s are generic constants independent of $t$, $u$,  ...   which may, however, vary from line to line. 
The (homogeneous) Morrey spaces over $\R^d$ modeled on $L^q(\R^d)$, $q\geqslant 1$, are defined by their norms 
\be
|\!\!| u|\!\!|_{M^s_q}\equiv \left(\sup_{R>0,\, x\in\R^d}R^{d(q/s-1)} 
\int_{\{|y-x|<R\}}|u(y)|^q\dy\right)^{1/q}<\infty.\label{hMor}
\ee  

\noindent 
\emph{Caution:} the notation for Morrey spaces used elsewhere might be different, e.g. $M^s_q$ is  denoted by $M^{q,\lambda}$ with $\lambda=dq/s$ in \cite{S}. 

The most frequent situation is when $q=1$ and we consider $M^s_1\equiv M^s$. The spaces  $M^{d(p-1)/\a}(\R^d)$ and more general $M^{d(p-1)/\a}_q(\R^d)$, $q>1$, are  critical in the study of equation \rf{f_lap}, see \cite{S}.  We refer the readers to \cite{B-SM,BZ-2,Lem} for analogous examples in chemotaxis theory.

 Integrals with no integration limits are meant to be calculated over the whole space $\R^d$. 
The asymptotic relation $f\approx g$ means that $\lim_{s\to\infty}\frac{f(s)}{g(s)}=1$ 
and $f\asymp g$ is used whenever $\lim_{s\to\infty}\frac{f(s)}{g(s)}\in (0,\infty)$.  
\section{Existence of global-in-time solutions below the singular solution}\label{s-glo}
\setcounter{equation}{0}

\subsection{Existence of the singular solution} 
We have the following 
\begin{proposition}\label{singC}
For $p>1+\frac{\alpha}{d-\alpha}$, there is a unique radial homogeneous nonnegative solution of the equation 
\be
(-\Delta)^{\nicefrac{\alpha}{2}}u_\infty=u_\infty^p,\ \ {\rm with}\ \ u_\infty(x)=s(\a,d,p)|x|^{-\frac{\a}{p-1}}\label{uC}
\ee 
with the constant 
\be
s(\a,d,p)= \left(
\frac{2^\a}{\Gamma\left(\frac{\a}{2(p-1)}\right)} 
 \frac{\Gamma\left(\frac{d}{2}-\frac{\a}{2(p-1)}\right) \Gamma\left(\frac{p\a}{2(p-1)}\right)}{\Gamma\left(\frac{d}{2}-\frac{p\a}{2(p-1)}\right)} 
\right)^{\frac{1}{p-1}}.\label{sC}
\ee
\end{proposition} 

\begin{proof}
By formula \rf{fr-lap} (see also \cite[(3.6)]{BKZ2}) for the fractional Laplacian, we have  for $d>\frac{p}{p-1}\a$, i.e. $p>1+\frac\a{d-\alpha}$
$$
\frac{\a}{p-1}+2\frac\a{2}=p\frac{\a}{p-1}.
$$ 
Thus, a multiple of the function $|x|^{-\a/(p-1)}$ is a solution of equation \rf{uC}. 
The value of $s(\a,d,p)$ is determined from the convolution identities \cite[(3.3)]{BKZ2}, see also formula \rf{lap-alpha-gamma}.  
The solution $u_\infty$ satisfies  $|\!\!| u_\infty|\!\!|_{M^{d(p-1)/\a}}=\frac{\s_d}{d-\frac{\a}{p-1}}  s(\a,d,p)$. 
\end{proof}

\begin{remark}
If $\a=2$ it is known that the exponent $p_{\rm sg}=1+\frac{2}{d-2}$ is critical for the existence of a positive, radial, distributional solution $u_\infty(x) = \left(\frac{2}{p-1}\left(d-2-\frac{2}{p-1}\right)\right)^{\frac{1}{p-1}}|x|^{-\frac{2}{p-1}}$ of the equation $\Delta u+u^p=0$ in $\bbfR^d$ (see for example \cite{QS} and references therein). In the case $\a\in (0,2)$ the analogue of this exponent is $1+\frac{\a}{d-\a}$, see Remark \ref{JL-exp}. 
For some smooth solutions of equation \rf{uC}, see \cite{BMW}. 

\end{remark}

\subsection{Small global-in-time solutions} 

Our  purpose  is to  prove two  global in time existence results, the first one  under smallness condition of the norm of critical homogeneous Morrey space $M^{d(p-1)/\a}_q(\R^d)$ with some $q>1$, and the second under the assumption that the initial data $u_0$ is below the singular solution $u_\infty$ (pointwisely). 

\begin{proposition}\label{glo1}
If $p>1+\frac{\a}{d}$, i.e. $\frac{d(p-1)}{\a}>1$, and the initial condition $u_0$ is sufficiently small in the sense of the norm of the homogeneous Morrey space $M^{d(p-1)/\a}_q(\R^d)$ with some $q>1$, then solution of problem \rf{f_lap}--\rf{u_0} is global in time. 
\end{proposition}

\begin{proof}
For $\a=2$,  if the norm $|\!\!| u_0|\!\!|_{M^{d(p-1)/\a}_q}$  (for a number $q>1$) is small enough, then a~solution of problem \rf{f_lap}--\rf{u_0} is global in time, see \cite[Proposition 6.1]{S} (as well as  a counterpart for the chemotaxis system cf. \cite[Th. 1]{B-SM}).  
Similarly as the former result, it  can be proved directly for $\a\in(0,2)$ (while the proof in \cite{S} for $\a=2$ was by contradiction) using the Picard iterations of the mapping 
$${\mathcal N}(u)(t)={\rm e}^{-t(-\Delta)^{\nicefrac{\alpha}{2}}}u_0+\int_0^t {\rm e}^{-(t-\tau)(-\Delta)^{\nicefrac{\alpha}{2}}}(|u|^{p-1}u)(\tau)\dta, 
$$
$u_{n+1}={\mathcal N}(u_n)$, $n=1,\,2,\, \dots$,   with $u_0\in M^{d(p-1)/\a}_q(\R^d)$ ($q>1$, $p>1+2q/d$) small enough. 
For useful estimates of the heat semigroup in Morrey spaces we refer the reader to \cite{G-M} and \cite[Th. 3.8]{Tay}.  They extend immediately to the estimate 
for $1< p_1< p_2< \infty$  
\be
|\!\!| {\rm e}^{-t(-\Delta)^{\nicefrac{\alpha}{2}}}f|\!\!|_{M^{p_2}_{q_2}}\leqslant Ct^{-d(1/p_1-1/p_2)/\alpha}|\!\!| f|\!\!|_{M^{p_1}_{q_1}}\label{polgrupa}
\ee
which holds with $q_2/q_1=p_2/p_1$ if $p_1<d$. 
Thus, we have 
$$
{\rm e}^{-\, \cdot\, (-\Delta)^{\nicefrac{\alpha}{2}}}:M^s_q(\R^d)\to {\mathcal Y}\equiv\{ u:\ \ \sup_{t>0}t^{-\beta}|\!\!| u(t)|\!\!|_{M^{rs/q}_r}<\infty\}
$$
with $ \beta=\frac{d}{\a}\left(1-\frac{q}{r}\right)=\left(1-\frac{q}{r}\right)\frac{1}{p-1}$. 
When either $p_1=1$ or $p_2=\infty$, the norm are $\|\, .\, \|_1$ and $\|\,.\,\|_\infty$, resp. 
The crucial estimate for the convergence of the Picard iterations is 
\be
\left\|\int_0^t{\rm e}^{-(t-\tau)(-\Delta)^{\nicefrac{\alpha}{2}}}(|u|^{p-1}u)(\tau)\dta\right\|_{M^{rd(p-1)/(\a q)}_r}\leqslant t^{-\beta}\xn u\xn^p,\label{nielin}
\ee
where $\xn u(t)\xn = \sup_{t>0} t^{(1-q/r)/(p-1)}|\!\!| u(t)|\!\!|_{M^{rd(p-1)/(\a q)}_r}$, 
and this follows since 
$$
\int_0^t(t-\tau)^{-\frac{d}{\a}(qp/rs-q/rs)}\tau^{-p\beta}\dta =Ct^{-\beta}
$$
for some constant $C>0$ since $-\frac{d}{\a s}\frac{q}{r}(p-1)-p\beta+1=-\frac{q}{r}+p\beta+1=-\beta$.

\bigskip
They are convergent in the norm $\xn\, .\, \xn$ 
for  $\max\{p,q\}<r<pq$, since $\xn {\rm e}^{-t(-\Delta)^{\nicefrac{\alpha}{2}}}u_0\xn <\infty$.  
Clearly, assumption \rf{ssub} in Th. \ref{glo2} implies $|\!\!|u_0|\!\!|_{ M^{d(p-1)/\a}} <\frac{\s_d}{d-\frac\a{p-1}}\e s(\a,d,p)$, and the conclusion \rf{subb} reads $|\!\!|u(t)|\!\!|_{M^{d(p-1)/\a}}<\frac{\s_d}{d-\frac\a{p-1}}\e s(\a,d,p)$. 

Then, usual bootstrapping arguments like in \cite[proof of Th. 15.2, p. 81]{QS} apply and $u(t)\in L^\infty$ follows for each $t>0$. The author is indebted to Philippe Souplet for this argument showing regularity. Note that this existence result remains valid for nonlinearities with behavior $|F(u)|\sim |u|^p$, $u\to 0$. 
\end{proof}

\begin{remark}
Let us recall that for $\a=2$ the number $p_F = 1+\frac{2}{d}$, called the Fujita exponent, borders the case of a finite-time blowup for all positive solutions (for $p\leqslant p_F$) and the case of existence of some global in time bounded positive solutions (if $p>p_F$) (see for example \cite{QS} and references therein). The exponent $1+\frac{\a}{d}$ is a counterpart of the Fujita exponent in the case $\a\in(0,2)$, see \cite{Su}. 
\end{remark}

In the next proposition $u_0$ is supposed to have asymptotics for large $|x|$ like $u_\infty(x)\asymp \frac{1}{|x|^2}$  but local singularities strictly weaker since $u_0\in M^{\ts} (\R^d)$ with some $\ts>s$ is assumed. 

\begin{proposition}\label{MsMts}
If $p>1+\frac{\a}{d}$, $s=\frac{d(p-1)}{\a}<\ts$ and  $u_0\in M^s(\R^d)\cap M^{\ts}(\R^d)$, then there exists $T>0$ and a local in time solution 
$$
 u\in {\mathcal X}_T\equiv {\mathcal C}([0,T];M^s(\R^d)\cap M^{\ts}(\R^d))\cap \left\{u:(0,T)\to L^\infty(\R^d):\ \ \sup_{0<t<T}t^\frac{d}{\a \ts}\|u(t)\|_\infty<\infty\right\}
$$
of problem \rf{f_lap}--\rf{u_0}. 
\end{proposition}

\begin{proof}
We follow the approach and notations in the proof of  Proposition \ref{glo1}. 
We will estimate the nonlinear operator ${\mathcal N}(u)$ in the norms of the spaces $M^s(\R^d)$, $M^{\ts}(\R^d)$ and ${\mathcal Y}_T\equiv  \left\{u:(0,T)\to L^\infty(\R^d):\ \ \sup_{0<t<T}t^\frac{d}{\a \ts}\|u(t)\|_\infty<\infty\right\}$. 

The first estimate is in the $M^s(\R^d)$ norm
\begin{align}
|\!\!|{\mathcal N}(u)(t)|\!\!|_{M^s}&\leqslant C\int_0^t |\!\!|u(\tau)|\!\!|_{M^s} \|u(\tau)\|_\infty^{p-1}\dta\nonumber\\ 
&\leqslant  C\int_0^t |\!\!|u(\tau)|\!\!|_{M^s} \tau^{-\frac{d}{\a \ts}(p-1)}\|u\|_{{\mathcal Y}_T}^{p-1}\dta\nonumber\\
&\leqslant CT^{1-\frac{d}{\a \ts}(p-1)}\sup_{t\in(0,T)}|\!\!|u(t)|\!\!|_{M^s}\|u\|_{{\mathcal Y}_T}^{p-1}\label{e1}
\end{align}
since $ \frac{d}{\a \ts}(p-1)< \frac{d}{\a s}(p-1)=1$.

The second estimate is in the $M^{\ts}(\R^d)$ norm 
\begin{align}
|\!\!|{\mathcal N}(u)(t)|\!\!|_{M^{\ts}}&\leqslant C\int_0^t (t-\tau)^{-\frac{d}{\a}\left(\frac{1}{s}-\frac{1}{\ts}\right)}|\!\!|u(\tau)|\!\!|_{M^s}\|u(\tau)\|_\infty^{p-1}\dta\nonumber\\ 
&\leqslant  C\int_0^t (t-\tau)^{-\frac{d}{\a}\left(\frac{1}{s}-\frac{1}{\ts}\right)}|\!\!|u(\tau)|\!\!|_{M^s} \tau^{-\frac{d}{\a \ts}(p-1)}\|u\|_{{\mathcal Y}_T}^{p-1}\dta\nonumber\\
&\leqslant CT^{1-\frac{d}{\a}\left(\frac{1}{s}-\frac{1}{\ts}\right)-\frac{d}{\a \ts}(p-1)}\sup_{t\in(0,T)}|\!\!|u(t)|\!\!|_{M^s}\|u\|_{{\mathcal Y}_T}^{p-1}\label{e2}
\end{align}
since $1-\frac{d}{\a}\left(\frac{1}{s}-\frac{1}{\ts}\right)-\frac{d}{\a \ts}(p-1)=1-\frac{1}{p-1}-\frac{d}{\a \ts}(p-2)>0$. 

The third estimate is 
\begin{align}
t^{\frac{d}{\a \ts}}\|{\mathcal N}(u)(t)\|_\infty &\leqslant  Ct^{\frac{d}{\a \ts}} \int_0^t(t-\tau)^{-\frac{d}{\a s}}\tau^{-{\frac{d}{\a \ts}(p-1)}}\dta \|u\|_{{\mathcal Y}_T}^{p}\nonumber\\
 &\leqslant  CT^\nu \|u\|_{{\mathcal Y}_T}^{p}\label{e3}
\end{align}
with $\nu=\frac{d}{\a \ts}-\frac{d}{\a s}-\frac{d}{\a \ts}(p-1)+1>0$. 
These bounds \rf{e1}--\rf{e3} lead in a standard way to the convergence of the Picard iterations for  initial data in  $M^{s}(\R^d)\cap M^{\ts}(\R^d)$ of arbitrary size and for sufficiently small $t>0$. 
\end{proof}

\subsection{Large global-in-time solutions}
The main result in this Section  is 
\begin{theorem}\label{glo2}
If  $\a\in(0,2)$, 
$p>1+\frac{\a}{d-\a}(>1+\frac{\a}{d})$, $u=u(x,t)$ is a solution of problem \rf{f_lap}--\rf{u_0}   with the initial data satisfying $u_0\in M^{\ts}(\R^d)$ for some $\ts>\frac{d(p-1)}{\alpha}$ and 
\be 
0\leqslant u_0(x)\leqslant u_\infty(x), \label{ssub}
\ee 
as well as 
\be
\lim_{|x|\to\infty}|x|^\ap u(x,t)=0 \ \ {\rm uniformly\ in}\ \  t\in(0,T).\label{ass-dec}
\ee
Then $u$ can be continued to a global-in-time solution which still satisfies the bound 
\be
0\leqslant u(x,t)\leqslant u_\infty(x).\label{ssubb} 
\ee 
\end{theorem}

Condition \rf{ssub} means that $u_0\in M^s(\R^d)$ with $s=\frac{d(p-1)}{\alpha}$. Thus, $\ts<s$ can be chosen as close to $\frac{d(p-1)}{\alpha}$ as we wish. Therefore, Proposition \ref{MsMts} on local-in-time solutions applies to those data. 

The result in Theorem \ref{glo2} is based on the following {\em restricted comparison principle}, see \cite[Th. 4.1, Th. 5.1]{BKZ2} for analogous albeit more complicated constructions for radially symmetric solutions of chemotaxis systems. 

\begin{proposition}\label{glo3}
For each $\delta\in(0,1)$ and each $K>0$ there exist $\g_0\in\left(0,\frac\alpha{p-1}\right)$ (independent of $K$ and sufficiently close to $\frac{\alpha}{p-1}$) 
such that every solution $u\in{\mathcal C}^1(\R^d\times(0,T])$ with the properties 
\be
\lim_{|x|\to\infty}|x|^\ap u(x,t)=0 \quad \textit{uniformly in}\quad   t\in(0,T), 
\label{infi}
\ee 
and the initial data satisfying 
\be
0\leqslant u_0(x)<\min\left\{\frac{K}{|x|^{\g_0}},\frac{\delta s(\alpha,d,p)}{|x|^{\frac{\alpha}{p-1}}}\right\}\equiv b(x)
,\label{sub}
\ee
satisfies the estimate 
\be
0\leqslant u(x,t)<b(x)\ \ \ {\rm for\ \ all\ \ }x\in\R^d\ \ {\rm and\ \ } 0< t\leqslant T.\label{subb} 
\ee
\end{proposition}

Once the comparison principle in Proposition \ref{glo3} is proved, the local-in-time solution constructed in Proposition \ref{MsMts}  can be  continued  onto some interval $[T,T+h]$, and further, step-by-step with the same $h>0$, onto the whole half-line $[0,\infty)$  to a global-in-time solution satisfying the bound $u(x,t)<b(x)$.  
Indeed, if inequality \rf{sub} holds, then $u\in M^{\ts}(\R^d)$ for some $\ts>\frac{d(p-1)}{\alpha}$, and clearly $u_0\in M^{\frac{d(p-1)}{\alpha}}(\R^d)$ by $0\leqslant u_0(x)\leqslant u_\infty(x)$. By Proposition \ref{MsMts}, the solution $u$ is locally in time in $L^\infty(\R^d)$, hence smooth by standard arguments. 

\begin{proof}[Proof of Theorem \ref{glo2}]
Approximating $u_0$ in Theorem \ref{glo2} by initial data \newline $u_{0k}(x)=\min\{u_0(x),(1-\nicefrac{1}{k})u_\infty(x)\}$ satisfying \rf{sub} with $\delta=1-\nicefrac{1}{k}$, $k=2,\, 3,\, 4, \dots$, we obtain a global-in-time solution via a monotonicity argument. This procedure is an adaptation of the monotone  approximation argument for the classical nonlinear heat equation in \cite{GV}. 
The first step is to show the property that $u_k$ increase with $k$. This can be done in standard way writing the equation for the difference $w$ of two approximating solutions $w=u_l-u_k$, $l>k$, as 
$$
w_t=-(-\Delta)^{\nicefrac{\alpha}{2}}w+\frac{u_l^p-u_k^p}{u_l-u_k}w.$$
 This is a linear equation of the type $w_t=-(-\Delta)^{\nicefrac{\alpha}{2}}w+V(x)w$ with $0\le V(x)\le C|x|^{-\alpha}$ considered in the next Section. The associated semigroup conserves positivity of the initial data $w_0$. 
Then, the pointwise monotone limit of $u_k$'s exists and satisfies equation \rf{f_lap} in the weak sense (using the Lebesgue dominated convergence theorem). 

\end{proof}

We sketch the proof of the comparison principle.

\begin{proof}[Proof of Proposition \ref{glo3}]
Let $u$ be a solution of problem \rf{f_lap}--\rf{u_0} for an initial data satisfying \rf{sub}. The proof of inequality \rf{subb} is by contradiction. 
Suppose that there exists $t_0\in(0,T]$ which is the first moment when $u(x,t)$ hits the barrier $b(x)$ defined in \rf{sub}. 
By {\em a~priori} ${\mathcal C}^1$ regularity of $u(x,t)$ and by property \rf{infi} the value of $t_0$ is well defined. 
Moreover, there exists $x_0\in \R^d$ satisfying $u(x_0,t_0)=b(x_0)$. 
Define the number 
\be
R_\#=\left(\frac{\delta s(\alpha,d,p)}{K}\right)^{1/\left(\g_0-\frac{\alpha}{p-1}\right)}\label{RR}
\ee
where two parts of the graph of the barrier $b$ meet. 
Let us consider the auxiliary  function  
\begin{equation}
\tilde u(x,t_0)=|x|^{\g}u(x,t_0)\ \ {\rm with}\ \ \g=  \left\{ \begin{array}{ll}  \label{gamma}
\frac{\alpha}{p-1} & \textrm{if \quad $|x_{0}|\ge R_\#$,}\\
\frac{d}{\ts}<\frac{\alpha}{p-1}& \textrm{if \quad $0\leqslant R_\#\le |x_{0}|$.} 
 \end{array}\right. 
 \end{equation} 
With this choice $\tilde u$ hits a {\em constant} part of the graph of the modified barrier $|x|^\g b(x)$. 

Let us compute according to formula \rf{fr-lap} for the fractional Laplacian 
\begin{align}
\frac{\partial}{\partial t}\tilde u(x_0,t)\big|_{t=t_0}&= 
|x_0|^{\g}(-(-\Delta)^{\nicefrac{\alpha}{2}} u+u^p)\nonumber\\
&=|x_0|^{\g}\left(-(-\Delta)^{\nicefrac{\alpha}{2}}(|x|^{-\g}\tilde u(x,t_0))\big|_{x=x_0} +  \left(|x_0|^{-\g}\tilde u(x_0,t_0)\right)^p\right)\nonumber\\
&=|x_0|^\g P.V.{\mathcal A}\int\frac{|x_0|^\g\tilde u(x_0-y,t_0)-|x_0-y|^\g\tilde u(x_0,t)}{|x_0|^\g|x_0-y|^\g|y|^{d+\a}}\dy\label{1.4}\\
&+|x_0|^{-(p-1)\g}\tilde u(x_0,t_0)^p\nonumber\\
&=|x_0|^\g\tilde u(x_0,t_0)P.V.{\mathcal A}\int\frac{1}{|y|^{d+\a}} \left(\frac{1}{|x_0-y|^{\g}}-\frac{1}{|x_0|^\g}\right) \dy\label{4.19}\\
&+|x_0|^\g P.V.{\mathcal A}\int\frac{\tilde u(x_0-y,t_0)-\tilde u(x_0,t_0)}{|x_0-y|^\g|y|^{d+\a}}\dy +|x_0|^{-(p-1)\g}\tilde u(x_0,t_0)^p\label{3.6}\\
&\leqslant \tilde u(x_0,t_0)\left(|x_0|^\g(-\Delta)^{\nicefrac{\alpha}{2}}\left(|x|^{-\g}\right)_{|_{x=x_0}}+|x_0|^{-(p-1)\g}\tilde u(x_0,t_0)^{p-1}\right)\label{maxi}
\end{align}
The passage from \rf{1.4} to \rf{4.19} is obvious, see an analogous reasoning in \cite[(4.19)]{BKZ2}. 
Formula \rf{4.19} follows by \rf{fr-lap} above. 
Formula \cite[(3.6)]{BKZ2} obtained for $\a+\gamma<d$ from \rf{fr-lap} and convolution identities for powers $|x|^{-\gamma}$, etc., \cite[(3.3)]{BKZ2}
\be
(-\Delta)^{\nicefrac{\alpha}{2}}\left(|x|^{-\gamma}\right) =2^\alpha\frac{\Gamma\left(\frac{d-\gamma}{2}\right)\Gamma\left(\frac{\alpha+\gamma}{2}\right)}{\Gamma\left(\frac{d-\alpha-\gamma}{2}\right)\Gamma\left(\frac{\gamma}{2}\right)}|x|^{-\alpha-\gamma}.
\label{lap-alpha-gamma}
\ee 
gives for $\gamma=\ap$ (so that $\a+\g<d$ by assumption $p>1+\frac{\a}{d-\a}$)
\[
(-\Delta)^{\nicefrac{\alpha}{2}}\left(|x|^{-\ap}\right) =2^\alpha\frac{\Gamma\left(\frac{d-\frac{\a}{p-1}}{2}\right)\Gamma\left(\frac{p\alpha}{2(p-1)}\right)}{\Gamma\left(\frac{d-\frac{p\alpha}{p-1}}{2}\right)\Gamma\left(\frac{\a}{2(p-1)}\right)}|x|^{-\frac{p\alpha}{p-1}}.
\label{lap-alpha-}
\]
We used in this passing from \rf{3.6} to \rf{maxi}. 
In the first case when $\g=\ap$, we recall  formula \rf{sC} for the constant $s(\a,d,p)$ which leads to 
$$\tilde u(x_0,t_0)^{p-1}=\delta^{p-1}s(\a,d,p)^{p-1}=\delta^{p-1} \frac{2^\a}{\Gamma\left(\frac{\a}{2(p-1)}\right)} 
 \frac{\Gamma\left(\frac{d}{2}-\frac{\a}{2(p-1)}\right) \Gamma\left(\frac{p\a}{2(p-1)}\right)}{\Gamma\left(\frac{d}{2}-\frac{p\a}{2(p-1)}\right)} $$
The passage from \rf{3.6} to \rf{maxi} in the middle term  
follows since $\tilde u(.,t_0)$ assumes its maximal value at $x_0$ and \rf{maxi} equals 
$$\tilde u(x_0,t_0)|x_0|^{-\a} s(\a,d,p)^{p-1}  (-1+\delta^{p-1})<0.$$  
The inequality $\frac{\partial}{\partial t}\tilde u(x_0,t)\big|_{t=t_0}<0$   contradicts the assumption that $\tilde u$ hits for the first time the {\em constant} level $\delta s(\alpha,d,p)$ at $t=t_0$.  

In the second case when $\g<\ap$ but close to $\ap$
\be
\tilde u(x_0,t_0)^{p-1}=\left(K|x_0|^{-\g}\right)^{p-1}\leqslant \left(\delta s(\alpha,d,p)|x_0|^{-\ap}\right)^{p-1}\label{up}
\ee 
by the definition of the barrier \rf{sub}. 
Since the coefficient in formula \rf{lap-alpha-gamma} is continuous when $\g\nearrow\ap$, so $(p-1)\g\nearrow\alpha$, the right-hand side of inequality \rf{maxi} is strictly negative for $\g$ close enough to $\ap$, i.e. for $\g\in(\g_0,\ap)$ with $\g_0=\g_0(\delta,\ap)$ independent of other parameters. 
\end{proof}

\begin{remark}\label{Noriko} 
This kind of result in Theorem \ref{glo2} is known for the classical nonlinear heat equation  (cf. \cite[Th. A]{GNW}) but the proof in \cite{B-bl} seems be somewhat novel. 
Similar pointwise arguments are powerful tools and, as such, they have been used in different contexts as e.g. fluid dynamics and  chemotaxis theory:  \cite{BGB,CC,CV,GB-OI},  and free boundary problems. 
If $u_0$ is radially symmetric and $u_0(x)< \delta u_\infty(x)$ for some $\delta<1$,  then the solution of \rf{f_lap}--\rf{u_0} exists globally in time, see \cite[Theorem 1.1]{M1} and also \cite[Remark 3.1(iv)]{S}. Related results are in \cite[Lemma 2.2]{M2}, and stability of the singular solution is studied in \cite{PY}.   
Results for not necessarily radial solutions starting either below or slightly above the singular solution $u_\infty$ are in \cite[Th. 10.4]{GV} (reported in \cite[Th. 20.5]{QS}),  and in \cite[Th. 1.1]{SW}. Note that solutions of the Cauchy problem for equation \rf{f_lap} for $\a=2$  in the latter case are nonunique. 

Here, these considerations have been  extended to the case of diffusion operators in equation \rf{f_lap} defined in a nonlocal way, similarly as it had been done in \cite[Theorem 2.5]{BKZ2} for radial solutions of the chemotaxis system with $\alpha\in(0,2)$. 
\end{remark}

\begin{remark}[decay estimates for global in time solutions]\label{decay-est}
The construction of solutions in Theorem \ref{glo2}
as perturbations of the solution ${\rm e}^{-t(-\Delta)^{\nicefrac{\alpha}{2}}}u_0$ to the linear equation 
$u_t+(-\Delta)^{\nicefrac{\alpha}{2}}u=0$ 
 leads to the asymptotic estimate 
\be
\|u(t)\|_\infty={\mathcal O}(t^{-1/(p-1)})\ \ {\rm  when}\ \ t\to\infty,\nonumber 
\ee
 exactly the same as for solutions to that equation 
with general positive initial data in $M^{d(p-1)/\a}(\R^d)$. 
Similarly as in the case of classical nonlinear heat equation \rf{f_lap} with $\a=2$, cf. \cite[Theorems 20.2, 20.6, 20.15]{QS}, one may expect better decay rate under stronger assumptions on the regularity of the initial data even for somewhat bigger size of data than in the proof of Theorem \ref{glo2}. 
\end{remark}

\section{Asymptotic behavior of solutions below the singular solution}\label{asym}
\setcounter{equation}{0}

We consider in this section behavior of solutions close to the singular solution $u_\infty$ but lying below it. {We use the same approach as in \cite{Pi1} for classical nonlinear heat equation.}

Introducing a new variable $w(x,t) =u_\infty (x) - u(x,t)$, where $ u = u (x, t) $ is a solution to \rf{f_lap}--\rf{u_0} the considered problem takes the form
\begin{align}
\label{h.e.1}  w_t & =- (-\Delta )^{\nicefrac{\alpha }{2}} w + s(\alpha, d, p)^{p-1}p |x|^{-\alpha}w - \Big( ( u_\infty -w)^p - u_\infty ^p + pu_\infty^{p-1}w\Big) ,\\
  \label{h.e.1.0} w(x,0)&=w_0(x),
\end{align}
where the constant $s(\alpha, d, p)$ is defined in formula \rf{sC}.
Note that the last term on the right-hand side of equation \rf{h.e.1} is nonpositive, namely
\begin{equation*} 
(u_\infty -w)^p-u_\infty ^p \geqslant -pu_\infty ^{p-1}w,
\end{equation*}
which is the direct consequence of the convexity of the function $f(s)=|s|^p$  
on $(0,\infty)$. Indeed, since the graph of the function $f$ lies above all of its tangents, we have $f(s-h)-f(s)\geqslant -f'(s)h$ for all $s$ and $h$ in $\bbfR$.

The proofs of our results are based on the following elementary observation. If $w$ is a~nonnegative solution of equation \rf{h.e.1} with the initial condition $w_0(x)\geqslant 0$, then 
\begin{equation*}
    0\leqslant w(x,t)\leqslant \se^{-tH}w_0(x)
\end{equation*}
with the operator $Hw= (-\Delta )^{\nicefrac{\alpha }{2}} w-s(\alpha, d, p) p |x|^{-\alpha}w$. Consequently, using the condition $0\leqslant u_0(x) \leqslant u_\infty (x)$ and the just mentioned comparison principle we can write 
\begin{align}
    \label{v inf and e H 0}  & 0\leqslant u_\infty (x) -u(x,t)\leqslant \se^{-tH}\big( u_\infty(x)-u_0(x) \big)\\
    \intertext{or, equivalently,    }
    \label{v inf and e H}    &u_\infty (x) -\se^{-tH}\big( u_\infty (x)-u_0(x) \big) \leqslant u(x,t)\leqslant u_\infty (x).
\end{align}

First, we concentrate on existence and properties of solutions to the linear initial value problem
\begin{align}
  \label{lin_w} w_t & =- (-\Delta )^{\nicefrac{\alpha }{2}} w + s(\alpha,d,p)^{p-1}p |x|^{-\alpha}w ,\\
  \label{lin_w_0} w(x,0)&=w_0(x).
\end{align}

\subsection{Linear fractional equation with the Hardy potential}

In this section we recall from \cite{BGJP} the estimate from above of the fundamental solution of the equation $u_t=- (-\Delta )^{\nicefrac{\alpha }{2}}u +\kappa 
|x|^{-\alpha}u$ for moderate values of $\kappa$, i.e. $\kappa\leqslant \frac{(2\pi)^\alpha}{c_\alpha}$ with the constant 
\begin{equation}\label{constant_c_alpha}
   c_{\alpha } = \pi^\alpha \left[ \Gamma \left(\frac{d-\alpha }{4}\right) / \Gamma \left( \frac{d+\alpha }{4} \right) \right]^2 
\end{equation} 
appearing in a fractional version of the Hardy inequality. 
 Following those arguments, we define the weights $\vf_\sigma (x,t) \in {\mathcal C}(\bbfR^d \setminus \{0\})$ as 
\begin{equation}\label{weights}
         \vf_\sigma (x,t)= 1+t^{\nicefrac{\sigma}{\alpha}}|x|^{-\sigma},
\end{equation}
where $\sigma \in (0, d-\alpha )$ satisfies the following equality
\begin{equation}\label{gamma_lambda}
    -2^{\alpha } +\lambda  \frac{\Gamma\left(\frac{\sigma}{2}\right)\Gamma\left(\frac{d}{2} - \frac{\sigma +\alpha}{2}\right)}{\Gamma\left(\frac{d}{2} - \frac{\sigma}{2}\right)\Gamma\left(\frac{\sigma + \alpha}{2}\right)}  = 0.
\end{equation}

 \begin{theorem}\label{kernel th}
       Let $\alpha\in(0,d\wedge 2)$, $Hu=- (-\Delta )^{\nicefrac{\alpha }{2}} u+s(\alpha,d,p)^{p-1} p |x|^{-\alpha}u$. 
Assume that $0 \leqslant s(\alpha, d,p)^{p-1} p\leqslant \tfrac{(2\pi )^\alpha}{c_\alpha }$ with  $c_\alpha$ defined in \rf{constant_c_alpha}.

      The semigroup $\se^{-tH}$ of the linear operators generated by $H$ can be written as the integral operator with a kernel denoted by $\se^{-tH}(x,y)$, namely
\begin{equation*}
         \se^{-tH}u_0 (x)=\int _{\bbfR ^n}\se^{-tH}(x,y)u_0(y) \ud y.
\end{equation*}    
      Moreover, there exists a positive constant $C$ such that for all $t>0$ and all $x,y \in \bbfR^d \setminus  \{0\}$
\begin{equation}\label{kernel}
         {0 \leqslant \se^{-tH}(x,y) \leqslant C\varphi_\sigma (x,t)\ \varphi_\sigma (y,t)\ G_\alpha (x-y, t)},
\end{equation}
\begin{equation*}
         0 \leqslant \se^{-tH}(x,y) \leqslant C(1+t^{\nicefrac{\sigma}{\alpha}}|x|^{-\sigma})(1+t^{\nicefrac{\sigma}{\alpha}}|y|^{-\sigma})\left(t^{-\nicefrac{d}{\alpha}}\wedge \frac{t}{|x-y|^{d+\alpha}}\right),
\end{equation*}
      where the functions $\vf_\sigma $ are defined in \rf{weights} and $G_{\alpha} (x-y, t)$ is the fractional heat kernel.
\end{theorem}

\begin{remark}\label{JL-exp}
Another important exponent for $\a=2$ is the Joseph-Lundgren exponent $p_{JL} = \frac{d-2\sqrt{d-1}}{d-4-2\sqrt{d-1}}$ which plays a crucial role in the study of stability of solutions to the classical nonlinear heat equation (see \cite{Pi1} and references therein). The analogue of this exponent for $\a\in(0,2)$ is the critical value of $p$ for which the assumption $s(\alpha, d,p)^{p-1} p\leqslant \tfrac{(2\pi )^\alpha}{c_\alpha }$ is fulfilled. 
Observe that the assumption 
\be 
s(\alpha, d,p)^{p-1} p\leqslant \tfrac{(2\pi )^\alpha}{c_\alpha} \label{pJL}
\ee
 is satisfied for certain $d$ (large) and $p$ (close to $\frac{d}{d-\alpha}>1$). 
This follows from the asymptotics of the expression  $s(\alpha, d,p)^{p-1} p$ 
 for $p\searrow \frac{d}{d-\alpha}$. Indeed, 
$$
\lim_{p\searrow \frac{d}{d-\alpha}}s(\alpha,d,p)^{p-1}p=0,
$$
thus there exists $p>\frac{d}{d-\alpha}$ satisfying inequality \rf{pJL}. 
Note that the exponent $p=\frac{d+\a}{d-\a}>1$ does not satisfy assumption \rf{pJL} since  according to  \cite[Remark 3]{BMW} \newline 
$s\left(\a,d,\frac{d+\a}{d-\a}\right)^{\frac{2\a}{d-\a}}=\frac{(2\pi)^\a}{c_\a}$.  
\end{remark}

The following theorem is the consequence of the estimates stated in formulas  \rf{kernel}. 

\begin{theorem}\label{w half l}
   Let the assumptions of Theorem \ref{kernel th} be valid {and $\sigma \in (0,d-\alpha)$ satisfy equation \rf{gamma_lambda}}. Assume that $p>1+\frac{\alpha}{d-\sigma}$.
   Suppose that there exist $b>0$ and $\ell \in \left(\frac{\alpha}{p-1}, d-\sigma\right)$ such that if a nonnegative function $w_0$ satisfies
\begin{alignat*}{2}
  &w_0(x)\leqslant b|x|^{-\frac{\alpha}{p-1}} &\quad &\textit{for} \quad |x|\leqslant 1 ,\\
  &w_0(x)\leqslant b|x|^{-\ell} &\quad &\textit{for}\quad  |x|\geqslant 1,
\end{alignat*}
  then
\begin{equation}\label{sup etH}
   \sup_{x\in \bbfR^d} \vf_\sigma ^{-1}(x,t) |\se^{-tH}w_0(x)|\leqslant Ct^{-\frac{\ell}{2}}
\end{equation}
holds  for a constant $C>0$ and all $t\geqslant 1$.
\end{theorem}

\begin{proof}
  First, for every fixed $x \in \bbfR^d $, we apply the estimate of the kernel $\se^{-tH}$ in Theorem \ref{kernel th} in the following way
\begin{equation*}
    \vf_\sigma ^{-1}(x,t) \big|\se^{-tH}w_0(x)\big| \leqslant C\int_{\bbfR^d} G_{\alpha} (x-y,t)\vf_\sigma (y,t)w_0(y) \ud y.
\end{equation*}
Next, we split the integral on the right-hand side into three parts $I_1(x,t)$, $I_2(x,t)$ and $I_3(x,t)$ according to the definition of the weights $\vf_\sigma $ and the assumptions on the function $w_0$. Let us begin with $I_1(x,t)$ 
\begin{equation*}
     \begin{split}
      I_1(x,t)&\equiv  C\int_{|y|\leqslant 1} G_\alpha (x-y,t)\vf_\sigma (y,t)w_0(y) \ud y \\
                 &\leqslant Cb t^{\nicefrac{\sigma }{\alpha}}\int_{|y|\leqslant 1} G_\alpha (x-y, t)|y|^{-\sigma -\frac{\alpha}{p-1}} \ud y 
      \leqslant Cbt^{-\nicefrac{(d-\sigma)}{\alpha}},
      \end{split}
\end{equation*} 
because $G_\alpha (x-y,t)$ is bounded by $C t^{-\nicefrac{d}{\alpha}}$ and the function $|y|^{-\sigma -\frac{\alpha}{p-1}}$ is integrable for $|y|\leqslant 1$ if $p>1+\frac{\alpha}{d-\sigma}$. 

We use the same argument to deal with 
\begin{equation*}
   \begin{split}
    I_2(x,t)&\equiv C\int_{1\leqslant |y|\leqslant t^{\nicefrac{1}{\alpha}}}G_\alpha (x-y,t)\vf_\sigma (y,t)w_0(y) \ud y\\
     &\leqslant Cbt^{\nicefrac{\sigma }{\alpha}} \int_{1\leqslant |y|\leqslant t^{\nicefrac{1}{\alpha}}} G_\alpha (x-y,t) |y|^{-\sigma -\ell} \ud y  
        \leqslant Cbt^{\nicefrac{\sigma -d}{\alpha}} \int_{1\leqslant |y|\leqslant t^{\nicefrac{1}{\alpha}}}|y|^{-\sigma -\ell} \ud y \\
     &\leqslant Cbt^{-\nicefrac{\ell}{\alpha}} + Cbt^{-\nicefrac{(d-\sigma)}{\alpha}}.
   \end{split}
\end{equation*}

Finally, we estimate 
\begin{equation*}
    \begin{split}
     I_3(x,t)&\equiv C\int_{|y|\geqslant t^{\nicefrac{1}{\alpha}}} G_\alpha (x-y, t)\vf_\sigma(y,t) w_0(y) \ud y \\
     &\leqslant Cb\int_{|y|\geqslant t^{\nicefrac{1}{\alpha}}} G_\alpha (x-y, t)|y|^{-\ell} \ud y \leqslant Cbt^{-\nicefrac{\ell}{\alpha}},
     \end{split}
\end{equation*}
using the inequality $1\leqslant \left(\frac{t^{\nicefrac{1}{\alpha}}}{|y|}\right)^\ell $ for $|y|\geqslant t^{\nicefrac{1}{\alpha}}$ and the identity $\int_{\bbfR^d} G_\alpha (x-y,t) \ud y=1$ for $t>0$, $x\in \bbfR^d $.
Since $\ell \in \left(\frac{\alpha}{p-1}, d-\sigma\right)$,   the proof of  \rf{sup etH} is completed.
\end{proof}

\begin{theorem}\label{limit e -tH th}
    Assume that  $|\cdot |^\sigma w_0 \in L^\infty (\bbfR^d )$ and $\lim_{|x|\rightarrow \infty }|x|^\sigma w_0 (x)=0$, {where $\sigma \in (0,d-\alpha)$ satisfies equation \rf{gamma_lambda}}.
   Then 
\begin{equation}\label{limit e -tH}
   \lim_{t\rightarrow \infty }t^{\nicefrac{\sigma }{\alpha}}\sup _{x\in \bbfR^d}\vf^{-1} _\sigma (x,t)\left|\se^{-tH}w_0(x)\right|=0.
\end{equation}
\end{theorem}

\begin{proof}
For every fixed $x\in \bbfR^d $ we use the estimate from Theorem \ref{kernel th} as follows
\begin{equation*}
   \vf_\sigma ^{-1}(x,t) \big| \se^{-tH}w_0(x)\big| \leqslant C\int_{\bbfR^d} G_\alpha (x-y,t)\vf_\sigma (y,t)w_0(y) \ud y.
\end{equation*}
We decompose the integral on the right-hand side according to the definition of $\vf_\sigma $ and we estimate each term separately. 
Substituting $y=zt^{\nicefrac{1}{\alpha}}$ and using the fact that $\vf_\sigma (x,t) \leqslant 2t^{\nicefrac{\sigma}{\alpha}}|x|^{-\sigma}$ if $|x|\leqslant t^{\nicefrac{1}{\alpha}}$ we obtain
\begin{equation*}
    \begin{split}
     I_1(x,t)&\equiv C\int_{|y|\leqslant t^{\nicefrac{1}{\alpha}}} G_\alpha (x-y,t) \left( \frac{t^{\nicefrac{1}{\alpha}}}{|y|} \right)^{\sigma }w_0(y) \ud y\\
     &=Ct^{-\nicefrac{\sigma }{\alpha}}\int_{|z|\leqslant 1 }G\bigg(\frac{x}{t^{\nicefrac{1}{\alpha}}}-z,t\bigg)|z|^{-2\sigma} \left|t^{\nicefrac{1}{\alpha}}z\right|^{\sigma }w_0\left(t^{\nicefrac{1}{\alpha}}z\right) \ud z .
     \end{split}
\end{equation*}
Hence, 
\begin{equation*}
    t^{\nicefrac{\sigma }{\alpha}}\sup_{x\in \bbfR^d }I_1(x,t) \rightarrow 0 \quad \textrm{as} \quad t\rightarrow \infty
\end{equation*}
follows by the Lebesgue dominated convergence theorem, because $G\left(\frac{x}{t^{\nicefrac{1}{\alpha}}}-z, 1\right)$ is bounded and the function $|z|^{-2\sigma }$ is integrable for $|z|\leqslant 1$. 
By the assumption imposed on $w_0$, given $\ve >0$ we may choose $t$ so large that
\begin{equation*}
   \sup_{|y|\geqslant \sqrt{t} }|y|^\sigma w_0(y)<\ve .
\end{equation*}

Now, using the inequality $1\leqslant \left( \frac{t^{\nicefrac{1}{\alpha}}}{|y|} \right)^\sigma $ for $|y|\geqslant t^{\nicefrac{1}{\alpha}}$, we obtain 
\begin{equation*}
    \begin{split} 
     I_2(x,t)&\equiv \int_{|y|\geqslant t^{\nicefrac{1}{\alpha}}} G_\alpha (x-y,t) w_0(y) \ud y 
     \leqslant t^{-\nicefrac{\sigma }{\alpha}}\int_{|y|\geqslant t^{\nicefrac{1}{\alpha}}} G_\alpha (x-y,t)|y|^{\sigma }w_0(y) \ud y \\
     &\leqslant \ve t^{-\nicefrac{\sigma }{\alpha}}\int_{|y|\geqslant t^{\nicefrac{1}{\alpha}}}G_\alpha (x-y, t) \ud y .
     \end{split}
\end{equation*}
Since $\int_{\bbfR^d } G(x-y,t )\ud y=1$ for all $t>0$, $x\in \bbfR^d $ and since $\ve >0$ is arbitrary, we get 
\begin{equation*}
     t^{\nicefrac{\sigma }{\alpha}} \sup_{x\in \bbfR^d} I_2(x,t) \rightarrow 0 \quad \textrm{as} \quad t\rightarrow \infty .
\end{equation*}
\end{proof}

Let us define the weighted $L^q(\R^d)$-norm as follows
\begin{equation*}
      \|f\|_{q,\vf_\sigma (t)}=\bigg( \int_{\bbfR^d} |f(x)\vf_\sigma ^{-1} (x,t)|^q \vf_\sigma ^2 (x,t) \ud x \bigg) ^{\frac{1}{q}} \quad \textrm{ for each}    \quad 1\leqslant q <\infty ,
\end{equation*}      
     and     
\begin{equation*}
    \|f\|_{\infty, \vf_\sigma (t)}=\sup_{x\in \bbfR^d} \vf_\sigma ^{-1}(x,t)|f(x)| \quad {\rm for }\quad q=\infty .
\end{equation*}
Note, that in particular for $q=2$, the norm $\| \cdot \|_{2, \vf_\sigma (t)} $ coincides with the usual $L^2$-norm on $\bbfR^d $. 

\begin{proposition}[hypercontractivity estimates]\label{hyperc}
Suppose that $1\leqslant q\leqslant \infty $. Then the following inequality holds true
\begin{equation}\label{w norm}
    \|\se^{-tH}w_0\|_{q,\vf _\sigma (t)}\leqslant Ct^{-\frac{d}{\alpha}(\frac{1}{r}-\frac{1}{q})} \|w_0\|_{r,\vf_\sigma (t)}
\end{equation}
for every $1\leqslant r\leqslant q\leqslant \infty$ and all $t>0$, {where weights $\vf_\sigma$ are defined in \rf{weights}}.
\end{proposition}

\begin{proof}
Using estimates of the kernel $\se^{-tH}$ we get
\begin{align*}
\left\|\vf_{\sigma} ^{-1} e^{-{tH }}\vf_{\sigma} w_0\right\|_{p,\vf_\sigma (t)}^p&=\int \limits_{\bbfR ^d}\vf_\sigma^{-p} (x,t)|e^{-{tH}}w_0(x)\vf_\sigma (x,t)|^p\vf_\sigma ^2(x,t) \ud x \\
 &=\int \limits_{\bbfR ^d}\vf_\sigma^{-p+2} (x,t)\left|\int \limits_{\bbfR ^d}e^{-{tH}}(x,y)\vf_\sigma^2 (y,t)w_0(y)\ud y\right|^p \ud x \\
&\leqslant C\int \limits_{\bbfR ^d}\vf_\sigma^{-p+2} (x,t)\left|\int \limits_{\bbfR ^d}G_\alpha (x-y, t)\vf_\sigma (x,t)\vf_\sigma ^2(y,t)w_0(y)\ud y\right|^p \ud x\\ 
&=C\int \limits_{\bbfR ^d} \vf_\sigma^2 (x,t) \left|\int \limits_{\bbfR ^d}G_\alpha (x-y,t)\vf_\sigma ^2(y,t) w_0(y)\ud y\right|^p \ud x.
\end{align*}
Observe that the weights are bounded by 
\begin{align}\label{observation}
  \vf_\sigma (x,t) &\leqslant 1 \quad \text{if}\quad |x|\geqslant t^{\nicefrac{1}{\alpha}},\\
  \nonumber
  \vf_\sigma (x,t) &\leqslant 2t^{\nicefrac{\sigma}{\alpha}}|x|^{-\sigma} \quad \text{if}\quad |x|\leqslant t^{\nicefrac{1}{\alpha}}.
\end{align} 
We split the integral into two terms 
\begin{align*}
I_1&=C\int \limits_{|x|\geqslant t^{\nicefrac{1}{\alpha}}}\vf_\sigma^2 (x,t)\int \limits_{\bbfR ^d}G_\alpha (x-y,t)\vf_\sigma  ^2(y,t)w_0 (y)\ud y|^p \ud x,\\
I_2&=C\int \limits_{|x|\leqslant t^{\nicefrac{1}{\alpha}}}\vf_\sigma ^2(x,t)\left|\int \limits_{\bbfR ^d}G_\alpha (x-y,t)\vf_\sigma  ^2(y,t)w_0(y)\ud y\right|^p \ud x,
\end{align*}
and applying \rf{observation} we get
\begin{align*} 
I_1&\leqslant C\int \limits_{|x|\geqslant t^{\nicefrac{1}{\alpha}}}\left|\int \limits_{\bbfR ^d}G_\alpha (x-y,t)\vf_\sigma  ^2(y,t)w_0(y)\ud y\right|^p \ud x\leqslant \int \limits_{\bbfR ^d}\left|\int \limits_{\bbfR ^d}G_\alpha (x-y,t)\vf_\sigma  ^2(y,t)w_0(y)\ud y\right|^p \ud x\\ 
&= C\left\| G_\alpha (t)w_0 \vf_\sigma  ^2\right\|_p ^p\ \leqslant \ C t^{-\nicefrac{d}{\alpha}(\frac{1}{q}-\frac{1}{p})p}\left\| w_0\vf_\sigma  ^2\right\|_q ^p
\end{align*}
 using estimates of the semigroup generated by the  fractional Laplacian. Moreover, applying the Young inequality and estimates of this  semigroup we arrive at
\begin{align*}
I_2&=C\int\limits_{|x|\leqslant t^{\nicefrac{1}{\alpha}}}t^{\nicefrac{\sigma}{\alpha}}|x|^{-2\sigma }\left|\int \limits_{\bbfR ^d}G_\alpha (x-y,t)\vf_\sigma ^2(y,t)w_0(y) \ud y\right|^p \ud x\\
&\leqslant Ct^{\nicefrac{\sigma}{\alpha}}\int \limits_{|x|\leqslant t^{\nicefrac{1}{\alpha}}}|x|^{-2\sigma}\|G_\alpha (t)w_0\vf_\sigma  ^2\|^p_{\infty }\ud x \\
&\leqslant Ct^{\nicefrac{\sigma}{\alpha}}\int \limits_ {|x|\leqslant  t^{\nicefrac{1}{\alpha}}}|x|^{-2\sigma }\| G_\alpha (t)\| ^p_{\frac{q}{q-1}}\| w_0\vf_\sigma  ^2\| ^p_q\ud x\leqslant Ct^{\nicefrac{2\sigma}{\alpha} -\frac{dp}{\alpha q}}\| w_0\vf_\sigma  ^2\| _q ^p\int \limits_{|x|\leqslant t^{\nicefrac{1}{\alpha}}}|x|^{-2\sigma }\ud x \\
&=C t^{-\frac{d}{\alpha}(\frac{1}{q}-\frac{1}{p})p}\left\| w_0\vf_\sigma  ^2\right\| _q ^p,
\end{align*}
which completes the proof.
\end{proof}

\subsection{Nonlinear equation}

Now we are in a position to state the result on convergence of solutions towards the singular steady state.
\begin{theorem}\label{half l}
   {Let $u=u(x,t)$ be a solution to problem \rf{f_lap}--\rf{u_0} constructed in Theorem \ref{glo2} with exponent $p$ satisfying assumption \rf{pJL} and $\sigma \in (0,d-\alpha)$ fulfill equation \rf{gamma_lambda}.} Assume that there exist constants $b>0$ and $\ell \in \big(\sigma , d-\sigma \big)$ such that 
\begin{equation*}
     u_\infty (x)-b|x|^{-\ell} \leqslant u_0(x)
\end{equation*}
  for all $|x|\geqslant 1$. Then
\begin{align}
     &\label{rel i} \sup_{|x|\leqslant t^{\nicefrac{1}{\alpha}} }|x|^\sigma \big( u_\infty (x) - u(x,t) \big) \leqslant Ct^{-\nicefrac{\ell -\sigma }{\alpha}} \\
     \intertext{and}
     &\label{rel ii} \sup_{|x|\geqslant t^{\nicefrac{1}{\alpha}}} \big( u_\infty (x) -u(x,t) \big) \leqslant Ct^{-\nicefrac{\ell }{\alpha}}.
\end{align}
hold for a constant $C>0$ and all $t\geqslant 1$.
\end{theorem}

\begin{remark}
For the classical nonlinear heat equation Pol\'a\v cik  and  Yanagida  \cite[Th. 6.1]{PY} proved a pointwise convergence of solutions to the singular steady state for $x\in \bbfR^d\setminus\{0\}$. 
Moreover, Fila and Winkler \cite{FW} showed the uniform convergence of solutions to the singular one on $\bbfR^d\setminus B_r(0)$, where $B_r(0)$ is the ball centered at the origin with the  radius $r$.
\end{remark}

\begin{proof}
It suffices to use inequality \rf{v inf and e H 0} and to estimate its right-hand side by Theorem \ref{w half l}.
\end{proof}

We can improve Theorem \ref{half l} for the limit exponent  $\ell=\sigma $ as follows.
\begin{theorem}\label{mth 2}
     {Let $u=u(x,t)$ be a solution to problem \rf{f_lap}--\rf{u_0} constructed in Theorem \ref{glo2} with exponent $p$ satisfying assumption \rf{pJL} and $\sigma \in (0,d-\alpha)$ fulfill equation \rf{gamma_lambda}.} Suppose that there exists a constant $b>0$ such that
\begin{equation*}
     u_\infty (x)-b|x|^{-\sigma  }\leqslant u_0 (x),
\end{equation*}
 and, moreover,
\begin{equation*}
      \lim_{|x|\rightarrow \infty }|x|^{\sigma }\big( u_\infty (x)-u_0(x) \big) = 0.
\end{equation*}    
   Then the relations 
\begin{align*}
     &\lim_{t\rightarrow \infty }\sup_{|x|\leqslant t^{\nicefrac{1}{\alpha}}}|x|^\sigma \big( u_\infty (x)-u(x,t)\big)=0\\
     \intertext{\it and }
    &\lim_{t\rightarrow \infty }t^{\nicefrac{\sigma }{\alpha}}\sup_{|x|\geqslant t^{\nicefrac{1}{\alpha}}}\big( u_\infty (x)-u(x,t)\big) =0
\end{align*}
hold. 
\end{theorem}

\begin{remark}
A similar result for the classical case, namely for $\alpha =2$, can be found in \cite{FWY1, FKWY,FW, FWY}, where authors proved estimates from below of the $L^\infty$-norm of solutions using matched asymptotics.
\end{remark}

\begin{proof}
As in the proof of Theorem \ref{half l}, it is sufficient to use \rf{v inf and e H 0} together with Theorem \ref{limit e -tH th}, substituting $w_0(x)=u_\infty (x)- u_0(x)$.
\end{proof}

\begin{corollary}\label{small b} 
   Under the assumptions of Theorem \ref{half l} and Theorem \ref{mth 2}, respectively if, moreover,  $b$ is sufficiently small, we obtain
\begin{align}
     \label{sup u}   &\|u(\cdot ,t)\|_\infty \geqslant C t^{\frac{\ell -\sigma}{\sigma (p-1)-\alpha}} \quad \textit{if } \quad \ell \in (\sigma , d-\sigma )\\
\intertext{\it for a constant $C>0$ and all $t\geqslant 1$,  and }
     \label{sup u 2}&\lim_{t\rightarrow \infty }\|u(\cdot ,t)\|_\infty =\infty  \quad \textit{if }\quad \ell=\sigma.
\end{align}
\end{corollary}

\begin{proof}
Since we have inequality \rf{v inf and e H}, it suffices to prove that
\begin{equation*}
    \sup_{x\in \bbfR^d}\left(u_\infty (x)-\se^{-tH}w_0(x)\right)\geqslant C(b) t^{\frac{\ell-\sigma }{\sigma (p-1)-\alpha}}
\end{equation*} 
for $w_0=u_\infty -u_0$. Hence, inequality \rf{sup etH} in Theorem \ref{w half l} enables us to write
\begin{equation*}
    u_\infty (x)-\se^{-tH}\big( u_\infty (x)- u_0(x) \big)\geqslant u_\infty (x)-Cb \vf_\sigma (x,t)t^{-\nicefrac{\ell}{\alpha}}
\end{equation*}
for all $x\in \bbfR^d \setminus \{0\}$ and $t>0$.
Next, using the explicit form of the weights $\vf_\sigma $, we define the function
\begin{equation*}
    F(|x|,t)=u_\infty (|x|)-Cb\vf_\sigma (x,t) t^{-\nicefrac{\ell }{\alpha}}=
        \begin{cases}
         s(\alpha,d,p)^{p-1}p|x|^{-\frac{\alpha}{p-1}}-Cbt^{\nicefrac{(\sigma -\ell)}{\alpha}}|x|^{-\sigma } & {\rm for }\  |x| \leqslant t^{\nicefrac{1}{\alpha}},\\
         s(\alpha,d,p)^{p-1}p|x|^{-\frac{\alpha}{p-1}}-Cbt^{-\nicefrac{\ell }{\alpha}} & {\rm for }\  |x|> t^{\nicefrac{1}{\alpha}}.
      \end{cases}
\end{equation*} 
An easy computation shows that the function $F$ attains its maximum at 
\begin{equation*}          
    |x|=C(b ) t^{ \nicefrac{(\sigma - \ell)}{\alpha}\frac{p-1}{\sigma (p-1)-\alpha}}, 
\end{equation*}    
and this is equal to
\begin{equation*}
   \max_{x\in \bbfR^d }F(|x|,t)=C(b) t^{\frac{\ell-\sigma }{\sigma (p-1)-\alpha}}
\end{equation*}   
for some constant $C(b)\geqslant 0$. Hence, we get \rf{sup u}.

To obtain \rf{sup u 2}, we use the result from Theorem \ref{limit e -tH th}. It follows from \rf{limit e -tH} that for every $\ve >0$ there exists $T>0$ such that 
\begin{equation*}
     \big| \se ^{-tH}w_0 (x) \big| \leqslant\ve \vf_\sigma  (x,t) t^{-\nicefrac{\sigma }{\alpha}}
\end{equation*}
for all $x \in \bbfR^d \setminus \{0\}$ and $t>T$.
Hence, by \rf{v inf and e H}, we have
\begin{equation*}
      u_\infty (x)-\se^{-tH}\big( u_\infty (x)- u_0(x) \big)\geqslant u_\infty (x)-C \ve \vf_\sigma (x,t)t^{-\nicefrac{\sigma }{\alpha}}.
\end{equation*} 
Now,  using again the explicit form of the weights $\vf_\sigma $, we consider the function 
\begin{equation*}
    G(|x|,t)=v_\infty (|x|)-Cb\vf_\sigma (x,t) t^{-\nicefrac{\sigma }{\alpha}}
    = \begin{cases}
         s(\alpha,d,p)^{p-1}p|x|^{-\frac{\alpha}{p-1}}-\ve |x|^{-\sigma } & {\rm for }\  |x| \leqslant t^{\nicefrac{1}{\alpha}},\\
         s(\alpha,d,p)^{p-1}p|x|^{-\frac{\alpha}{p-1}}-\ve t^{-\nicefrac{\sigma }{\alpha}} & {\rm for }\  |x|> t^{\nicefrac{1}{\alpha}}.
       \end{cases}
\end{equation*}
Elementary computations give us that the function $G$ attains its maximum at
\begin{equation*}          
    |x|=c \ve^{ -\frac{p-1}{\sigma (p-1)-\alpha}}
\end{equation*}    
and 
\begin{equation*}
   \max_{x\in \bbfR^d }G(|x|,t)=C\ve^{-\frac{\alpha}{\sigma (p-1)-\alpha}}
\end{equation*}   
for some constant $C\geqslant 0$. Since $\sigma>\frac{\alpha}{(p-1)}$, we see that the maximum of the function $G$ diverges to infinity if $\ve $ tends to zero. This completes the proof of \rf{sup u 2}.
\end{proof}   

Our next goal is to prove the asymptotic stability of the singular solution  $u_\infty $ in the Lebesgue space $L^2 (\bbfR^d)$.

\begin{theorem}\label{stab.2}
{Let $u=u(x,t)$ be a solution to problem \rf{f_lap}--\rf{u_0} with exponent $p$ satisfying assumption \rf{pJL} and $\sigma \in (0,d-\alpha)$ fulfill equation \rf{gamma_lambda}.}
\begin{itemize}
    \item[i)]  Suppose that $u_\infty -u_0 \in L^1(\bbfR^d )$ and $|\cdot |^{-\sigma } (u_\infty -u_0)\in L^1(\bbfR^d)$. Then 
\begin{equation}\label{1:  L1}
    \| u_\infty  - u( t) \|_2 \leqslant Ct^{-\frac{d}{2\alpha}}\| u_\infty  -u_0  \|_1 + Ct^{-\frac{d-2\sigma }{2\alpha}}\| |\cdot |^{-\sigma }(u_\infty -u_0 )\|_1 .
\end{equation}    
    \item[ii)] Suppose that $u_\infty -u_0 \in L^2 (\bbfR^d )$. Then
\begin{equation*}
    \lim_{t \rightarrow \infty } \|u_\infty -u(t) \|_2 =0.  
\end{equation*}
\end{itemize}
\end{theorem}

\begin{proof}[Proof of Theorem \ref{stab.2} i)]
According to estimates \rf{v inf and e H 0} it is enough to estimate the $L^2$-norm of the expression $\se^{-tH}w_0$ for every $w_0$ satisfying two conditions: $w_0\in L^1 (\bbfR^d )$ and $| \cdot |^{-\sigma }w_0 \in L^1 (\bbfR^d )$. Applying \rf{w norm} with $q=2$, $r=1$, and using the definition of the functions $\vf_\sigma (x,t)$, we may write
\begin{equation*}
     \begin{split}
     \|\se^{-tH} w_0\|_2 &\leqslant Ct^{-\frac{d}{2\alpha}}\|w_0\|_{1, \vf_\sigma (t)} =Ct^{-\frac{d-2\sigma }{2\alpha}}\int_{|x|\leqslant t^{\nicefrac{1}{\alpha}}}w_0(x)|x|^{-\sigma }\ud x\\
                                   &+Ct^{-\frac{d}{2\alpha}}\int_{|x|\geqslant t^{\nicefrac{1}{\alpha}}}w_0(x)\ud x  \leqslant Ct^{-\frac{d-2\sigma }{2\alpha}}\| w_0 |\cdot |^{-\sigma }\|_1 + Ct^{-\frac{d}{2\alpha}}\| w_0\|_1.
     \end{split}
\end{equation*} 
This establishes formula \rf{1: L1}.    
\end{proof}

\begin{proof}[Proof of Theorem \ref{stab.2} ii)]
Again, by inequalities \rf{v inf and e H 0},  we only need to show that  
\begin{equation*}
     \lim_{t\rightarrow \infty } \|\se^{-tH}w_0\|_2 =0
\end{equation*}
for each $w_0 \in L^2 (\bbfR^d)$. Hence, for every $\ve >0$ we choose 
$\psi \in C_c^\infty ( \bbfR^d )$ such that $\|w_0 -\psi\|_2<\ve $. 
Using  the triangle inequality first and next estimate \rf{w norm} with $q=2$ and $r=2$, we obtain
\begin{equation*}
     \begin{split}
     \| \se^{-tH}w_0 \|_2 &\leqslant \|\se^{-tH}(w_0 -\psi )\|_2 + \| \se^{-tH} \psi \|_2 \\
     &\leqslant C\ve + \| \se^{-tH} \psi \|_2.
     \end{split}
\end{equation*} 
Since the second term on the right-hand side converges to zero as $t\rightarrow \infty $ by the first part of Theorem \ref{stab.2}, we get
\begin{equation*}
     \limsup_{t\rightarrow \infty }\| \se^{-tH}w_0\|_2 \leqslant C\ve.
\end{equation*}     
This completes the proof of Theorem \ref {stab.2} ii), because $\ve >0$ can be arbitrarily small.  
\end{proof}
\subsection{Decay of solutions}

We prove an asymptotic result for solutions considered in Theorem \ref{glo2} 
\begin{theorem}\label{zero_conv}
Let $u$ be a solution of problem \rf{f_lap}--\rf{u_0} with $u_0\in L^1_{\vf_\sigma(t)} (\bbfR^d)$ satisfying the assumptions of Theorem \ref{glo2}. Then
\begin{equation*}
  \lim_{t\to \infty } \| u(t)\|_{q,\vf_\sigma(t)} = 0 
\end{equation*}
for each $1\leqslant q\leqslant \infty$ holds.
\end{theorem}

\begin{proof}
By Proposition \ref{glo3} we have $u(x,t) <\delta u_\infty (x)$ for $t>0$, hence
\begin{equation*}
  u_t =  - (-\Delta )^{\nicefrac{\alpha }{2}} u+ |u|^{p-1}u < - (-\Delta )^{\nicefrac{\alpha }{2}} u+ (\delta s(\a, d,p))^{p-1}|x|^{-\a}u .
\end{equation*}
If $(\delta s(\a, d,p))^{p-1}\leqslant \frac{(2\pi)^\a}{c_\a}$, where $c_\a$ is defined in \rf{constant_c_alpha}, we get the estimate
\begin{equation*}
  u(x,t) \leqslant e^{-tH_\delta}u_0(x),
\end{equation*}
where $H_\delta u = - (-\Delta )^{\nicefrac{\alpha }{2}} u+(\delta s(\alpha,d,p))^{p-1} |x|^{-\alpha}u$. Using Proposition \ref{hyperc} for $1\leqslant q \leqslant \infty $ and $r=1$ we obtain
\begin{equation*}
  \|u(t) \|_{q,\vf_\sigma(t)} \leqslant Ct^{-\frac{d}{\alpha}(1-\frac{1}{q})} \|u_0\|_{1,\vf_\sigma (t)},
\end{equation*}
which completes the proof of Theorem \ref{zero_conv}.
\end{proof}

\section{Complements and comments}

A sufficient condition for blowup of solutions of equation \rf{f_lap} with $p>1+\frac{\a}{d}$ 
\be
T^\frac{1}{p-1}\left\|{\rm e}^{-t(-\Delta)^{\nicefrac{\a}{2}}}u_0\right\|_\infty>C_{\a,d,p}\label{sblo}
\ee 
derived  in \cite{Su}  has been interpreted in \cite{B-bl-a} as 
\be
\|u_0\|_{M^{d(p-1)/\a}}>m_{\a,d,p}\label{mblo}
\ee 
for some $m_{\a,d,p}>0$. 
Indeed, we have  equivalence 
\be
\sup_{t>0}t^\g\left\|{\rm e}^{-t(-\Delta)^{\nicefrac{\alpha}{2}}}u\right\|_\infty<\infty\ {\rm if\ and\ only\ if\ } u\in B^{-\g\a}_{\infty,\infty}(\R^d)\label{Bes}
\ee
where $B^{-\kappa}_{\infty,\infty}$ is the homogeneous Besov space of order $-\kappa<0$. The above condition \rf{Bes} is for $u\ge 0$ equivalent to $u\in M^{\frac{d}{\a\g}}(\R^d)$, the Morrey space of order $\frac{d}{\a\g}$,  \cite[Prop. 2B)]{Lem} for $\a=2$ and a slight modification of  \cite[Sec. 4, proof of Prop. 2]{Lem2} for $\a\in(0,2)$.

These are counterparts of results in \cite[Remark 7, \, Theorem 2]{B-bl} for the classical nonlinear heat equation. 
Together with results of Section \ref{s-glo}, this  leads to the following partial \, {\em dichotomy } result, similarly as was in \cite[Corollary 11]{B-bl}

\begin{corollary}[dichotomy]\label{dich}
There exist two positive constants $c(\a,d,p)$ and $C(\a,d,p)$ such that if $p>1+\frac{\a}{d}$   then 
\begin{itemize}
\item[(i)]
 $|\!\!| u_0|\!\!|_{M^{d(p-1)/\a}_q}<c(\a,d,p)$ for some $q>1$, $q<\frac{d(p-1)}{\a}$, implies that problem \rf{f_lap}--\rf{u_0} has a~global in time, smooth  solution satisfying the time decay estimate 
 $\|u(t)\|_\infty={\mathcal O}\left(t^{-1/(p-1)}\right)$.   

\item[(ii)]
  $|\!\!| u_0|\!\!|_{M^{d(p-1)/\a}}>C(\a,d,p)$ implies that each nonnegative solution of problem \rf{f_lap}--\rf{u_0} blows up in a finite time.
\end{itemize}   
\end{corollary}

Of course, there are many interesting behaviors of solutions (and still open questions) for the initial data of intermediate size satisfying 
$$
c(\a,d,p)\leqslant |\!\!| u_0|\!\!|_{M^{d(p-1)/\a}}\leqslant C(\a,d,p),
$$ 
and/or suitable pointwise estimates comparing the initial condition  $u_0$ with the singular solution $u_\infty$.

It is of interest to compare these constants  $c(\a,d,p)$ and $C(\a,d,p)$ with the Morrey space norm $|\!\!| u_\infty|\!\!|_{M^{d(p-1)/\a}}$ of the singular stationary solution $u_\infty>0$ of \rf{f_lap} in \rf{uC}--\rf{sC}.

Sufficient conditions for local-in-time existence of solutions are intimately connected with the problem of initial traces, i.e. a characterization of $u_0$ such that $u(t)$ tends to $u_0$ weakly as $t\to 0$ for a given nonnegative local solution of equation \rf{f_lap}.
Conditions on such $u_0$'s, roughly speaking, mean that local singularities are weaker than a multiple of $|x|^{-\a/(p-1)}$ (or $\lim_{x\to x_0} |x-x_0|^{\a/(p-1)}u_0(x)\leqslant J(\a,d,p)$ for a universal constant $J(\a,d,p)>0$) so that the size of $u_0$ in $M_{\rm loc}^{d(p-1)/\a}(\R^d)$ is universally bounded.

\begin{remark}[initial traces]\label{traces} 
General results on  the existence of initial traces (i.e. $u_0$'s) for arbitrary {\em nonnegative} weak solutions of equation \rf{f_lap} ($u=u(t)$ defined on $(0,T)$)  can be inferred from  \cite{ADB,BP} using local moments like those in \cite{BKZ2} with the weight functions $(1-|x|)_+^{1+\a/2}$ used in the analysis of nonlocal problems of chemotaxis. 
In particular,   estimates analogous to those in \cite[(1.4), Prop. 4.3 on p. 380]{ADB} show that a~necessary condition for the existence of a local in time solution reads: $u_0\in M^{d(p-1)/\a}_{\rm loc}(\R^d)$, and each nonnegative solution satisfies $u(t)\in  M^{d(p-1)/\a}_{\rm loc}(\R^d)$ uniformly on $(0,T)$.  
Here, $u_0\in M^{d(p-1)/\a}_{\rm loc}(\R^d)$ means:  
$\limsup_{R\to 0,\, x\in\R^d}R^{\a/(p-1)-d} \int_{\{|y-x|<R\}}|u_0(y)|\dy<\infty$. 
Thus, $M^{d(p-1)/\a}(\R^d)$ is, in a sense,  close to be  the optimal space for the local in time solvability of problem \rf{f_lap}--\rf{u_0} in the class of nonnegative solutions. 

If in both (i) and (ii) of Corollary \ref{dich} on dichotomy there were a single functional norm $\bar \ell$ instead of those of $M^{d(p-1)/\a}_q(\R^d)$ with some $q>1$ and $M^{d(p-1)/\a}(\R^d)$ (both spaces are critical) this will be unique,   
up to equivalence (a~personal communication of Philippe Souplet). 
Such a norm is called the {\em dichotomy norm}.  
Note that, however,  problem \rf{f_lap}--\rf{u_0} is not well posed in the critical space $M^{d(p-1)/\a}(\R^d)$, 
similarly to the case of radial solutions of the parabolic-elliptic Keller-Segel system studied in \cite{Lem,BKZ2} with  $M^{d/2}(\R^d)$ data as well as  for the fractional Keller-Segel  system with  $M^{d/\a}(\R^d)$ data  in \cite{B-book}.
Namely, there is no continuity of solutions of the Cauchy problem \rf{f_lap}--\rf{u_0} with respect to the initial data in this norm, see the following Remark \ref{disc-data}. 
In view of the result in \cite{CZ}, one can barely   expect that this $\bar \ell$ would be the norm in $M^{d(p-1)/\a}(\R^d)$ based solely on $L^1_{\rm loc}$ properties of functions. 
\end{remark}

\begin{remark}[solutions may depend discontinuously on the initial data in $M^{d(p-1)/\a}$]\label{disc-data} 
 \ 
\noindent 
If the condition $\limsup_{R\to 0}R^{\a/(p-1)-d} \int_{\{|y-x|<R\}}|u_0(y)|\dy>K(\a,d,p)$ is satisfied for some $x\in\R^d$ and a constant $K(\a,d,p)\geqslant C(\a,d,p)>0$, then solutions are not continuous with respect to the initial data at $u_0$; in fact, the existence times of approximating solutions tend to $0$ when the initial data are cut: $\un_{\{|y-x|>R_n\}}u_0$, $R_n\to 0$ as $n\to \infty$.  
This can be inferred from the sufficient condition for blowup and the estimate of the existence time for solutions, see analogous arguments in \cite{BKZ2,BZ-2}. 
\end{remark}


\end{document}